\documentclass[11pt]{article}
\usepackage{amsmath,amssymb,amsthm}
\usepackage{enumitem}
\usepackage[affil-it]{authblk}
\usepackage{a4wide}
\usepackage{color}
\usepackage{hyperref}
\usepackage{xfrac}

\newtheorem{thm}[equation]{Theorem}
\newtheorem{cor}[equation]{Corollary}
\newtheorem{prop}[equation]{Proposition}
\newtheorem{lemma}[equation]{Lemma}
\theoremstyle{definition}
\newtheorem{defn}[equation]{Definition}
\newtheorem{remark}[equation]{Remark}
\newtheorem{exam}[equation]{Example}
\numberwithin{equation}{section}

\newcommand{\FF}{\mathbb{F}}  
\newcommand{\ZZ}{\mathbb{Z}}

\newcommand{\Z}{\mathsf{Z}}

\newcommand\chara{\mathsf{char}}

\newcommand{\pb}[1]{\left\{ #1\right\}}
\newcommand{\seq}[1]{\left( #1\right)}

\newcommand\degg{\, \mathsf{deg} \,}

\newcommand\DEg{\operatorname{\underline{\mathsf{deg}}}}

\newcommand{\hqfg}[1][f,g]{\mathcal{H}_q(#1)}
\newcommand{\hqqfg}[1][f',g']{\mathcal{H}_{q'}(#1)}
\newcommand{\hh}{\mathcal{H}}

\newcommand{\autt}[1]{\mathsf{Aut}(#1)}

\title{Structure and isomorphisms of quantum generalized Heisenberg algebras}

\author{Samuel A.\ Lopes\thanks{Partially supported by CMUP, which is financed by national funds through FCT---Funda\c c\~ao para a Ci\^encia e a Tecnologia, I.P., under the project with reference UIDB/00144/2020.}\ }

\author{Farrokh Razavinia\thanks{Supported by FCT, through the grant PD/BD/142959/2018, under POCH funds, co-financed by the European Social Fund and Portuguese National Funds from MEC.}}
\affil{CMUP, Departamento de Matem\'atica, Faculdade de Ci\^encias, Universidade do Porto, Rua do Campo Alegre s/n, 4169--007 Porto, Portugal.}

\date{}

\begin{document}

\maketitle

\begin{abstract}
In~\cite{LR20arXiv} we introduced a new class of algebras, which we named \textit{quantum generalized Heisenberg algebras} and which depend on a parameter $q$ and two polynomials $f,g$. We have shown that this class includes all generalized Heisenberg algebras (as defined in \cite{CRM01} and \cite{LZ15}) as well as generalized down-up algebras (as defined in \cite{BR98} and \cite{CS04}), but the parameters of freedom we allow give rise to many algebras which are in neither one of these two classes (if $q\neq 1$ and $\degg f>1$). Having classified their finite-dimensional irreducible representations in~\cite{LR20arXiv}, in this paper we turn to their classification by isomorphism, the description of their automorphism groups and the study of ring-theoretical properties like Gelfand-Kirillov dimension and being Noetherian. 
\newline\newline
\textbf{MSC Numbers (2020)}: Primary 16W20; Secondary 17B81, 16S80, 16T20.
 \hfill \newline
\end{abstract}

\section{Introduction}\label{S:intro}

This paper continues the study of a new class of algebras introduced in \cite{LR20arXiv} and named \textit{quantum generalized Heisenberg algebras} (qGHA, for short), as they can be seen simultaneously as deformations and as generalizations of the generalized Heisenberg algebras appearing in \cite{CRM01} and profusely studied thenceforth in the physics literature (see e.g.~\cite{CRMR13}, \cite{BBH11}, \cite{BCG18} and the references therein). In the mathematics literature, generalized Heisenberg algebras were studied mainly in \cite{LZ15}, \cite{LMZ15} and \cite{sL17}. For an overview of their relevance in mathematical physics see the introductory section in \cite{LZ15}.

Our main motivation for introducing a generalization of this class, besides providing a broader framework for the investigation of the possible underlying physical systems, comes from the observation in \cite{sL17} that the classes of generalized Heisenberg algebras and (generalized) down-up algebras intersect (see the seminal paper \cite{BR98} on down-up algebras and also \cite{CS04}), although neither one contains the other. The other interesting feature of our study comes from the fact that quantum generalized Heisenberg algebras are generically non-Noetherian although they resemble and are related to deformations of enveloping algebras of Lie algebras.

\begin{defn}
Let $\FF$ be an arbitrary field and fix $q\in\FF$ and $f,g\in\FF[h]$. The quantum generalized Heisenberg algebra (qGHA, for short), denoted by $\hqfg$, is the $\FF$-algebra generated by $x$, $y$ and $h$, with defining relations:
\begin{equation}\label{E:intro:def:qgha}
hx=xf(h), \quad yh=f(h)y, \quad yx-qxy=g(h).
\end{equation} 
\end{defn}

The main results in this paper are Proposition~\ref{P:noetherian}, which characterizes the Noetherian quantum generalized Heisenberg algebras (compare \cite{KMP99} and \cite{CS04} for (generalized) down-up algebras and \cite{sL17} for the generalized Heisenberg algebras $\hh_1(f, f-h)$), Theorem~\ref{T:iso}, which classifies the algebras $\hqfg$ by isomorphism type (compare \cite{CM00} for down-up algebras and \cite{LZ15} which solves this problem for the generalized Heisenberg algebras $\hh_1(f, f-h)$) and Theorems~\ref{Thm2} and~\ref{Thm3}, which describe the structure of the automorphism group of $\hqfg$ (compare~\cite{CL09} for generalized down-up algebras and \cite{sL17} for the generalized Heisenberg algebras $\hh_1(f, f-h)$).

\subsection{Examples of quantum generalized Heisenberg algebras}\label{SS:intro:eg}

The generalized Heisenberg algebras from \cite{LZ15} are precisely the qGHA with $q=1$ and $g=f(h)-h$, i.e.\ the algebras of the form $\hh_1(f, f-h)$. Let us consider more general examples.

For parameters $\alpha, \beta, \gamma\in\FF$, the down-up algebra $A(\alpha, \beta, \gamma)$ was defined by Benkart and Roby in \cite{BR98} as the unital associative algebra with generators $d$ and $u$ and defining relations:
\begin{equation*}
d^2 u = \alpha dud + \beta ud^2 + \gamma d\quad \mbox{and}\quad
du^2 = \alpha udu + \beta u^2 d + \gamma u. 
\end{equation*}
In \cite{CS04}, Cassidy and Shelton generalized this construction and introduced the generalized down-up algebra $L(v,r,s,\gamma)$ as the unital associative algebra generated by $d$, $u$ and $h$ with defining relations
\begin{equation*}
dh-rhd +\gamma d =0, \quad \quad hu-ruh +\gamma u=0 \quad \mbox{and}\quad du-sud +v(h)=0,
\end{equation*}
where $r, s, \gamma\in\FF$ and $v\in\FF[h]$.
Generalized down-up algebras include all down-up algebras, as long as the polynomial $h^2-\alpha h-\beta$ has roots in $\FF$.  Moreover, the following are generalized down-up algebras: the algebras \textit{similar to the enveloping algebra of $\mathfrak{sl}_2$} defined by Smith \cite{spS90}, Le Bruyn conformal $\mathfrak{sl}_2$ enveloping algebras \cite{lLB95}, and Rueda's algebras \textit{similar to the enveloping algebra of $\mathfrak{sl}_2$} \cite{sR02}.

We have observed in~\cite{LR20arXiv} that the class of generalized down-up algebras coincides with the class of quantum generalized Heisenberg algebras $\hqfg$ such that $\degg f\leq 1$.

\begin{prop}[{\cite{LR20arXiv}}]\label{P:intro:gduaRqgha}
Let $r, s, \gamma\in\FF$ and $v\in\FF[h]$. Then the generalized down-up algebra $L(v,r,s,\gamma)$ is isomorphic to the quantum generalized Heisenberg algebra $\hh_s(rh-\gamma, -v)$. In particular, the down-up algebra $A(\alpha, \beta, \gamma)$ is isomorphic to the quantum generalized Heisenberg algebra $\hh_s(rh+\gamma, h)$, where $\alpha=r+s$ and $\beta=-rs$.

Conversely, any quantum generalized Heisenberg algebra $\hqfg$ such that $f(h)=ah+b$, with $a, b\in\FF$, is a generalized down-up algebra of the form $L(-g,a,q,-b)$.
\end{prop}

As a reciprocal to the above we shall see in Corollary~\ref{gduaimpdf1} that if a quantum generalized Heisenberg algebra $\hqfg$ is isomorphic to a generalized down-up algebra, then necessarily $\degg f\leq1$.

\subsection{Organization of the paper}\label{SS:intro:org}

In Section~\ref{S:basic} we review the basic properties of qGHA. By using an appropriate filtration and results on Gelfand-Kirillov dimension, we are able to prove in Corollary~\ref{gduaimpdf1} that if $\degg f>1$ then $\hqfg$ is not isomorphic to a generalized down-up algebra. This divides the class of qGHA into two natural subclasses: if $\degg f\leq1$ we get all generalized down-up algebras, which have been extensively studied from many points of view; if $\degg f>1$ we get algebras which are non-Noetherian domains (as long as $q\neq 0$) and which, in spite of appearing to be of a similar nature, have not been yet studied in depth, as far as we know.

In Section~\ref{S:noetherian} we characterize the Noetherian quantum generalized Heisenberg algebras. While it is well known that for generalized down-up algebras being Noetherian is equivalent to being a domain (\cite{KMP99}, \cite{CS04}), we see that within our wider class of algebras this correspondence no longer holds as for $q\neq 0$ and $\degg f>1$ the algebra $\hqfg$ will be a non-Noetherian domain.

The isomorphism problem for quantum generalized Heisenberg algebras is tackled in Section~\ref{S:class} and it will be seen that the isomorphism relation can be phrased in very concrete geometric terms, very much like in~\cite{BJ01}. It will follow in particular that, in case $q\neq 0$ and $\degg f>1$, the parameter $q$, as well as the integers $\degg f$ and $\degg g$, are invariant under isomorphism, showing that qGHA are indeed a vast generalization of generalized Heisenberg algebras and generalized down-up algebras. 

In terms of automorphism groups, which we study in Section~\ref{S:automorphisms}, an interesting phenomenon occurs. Although, as long as either $\chara(\FF)=0$ or $\chara(\FF)>\degg f$, the automorphism group of a quantum generalized Heisenberg algebra $\hqfg$ with $q\neq 0$ and $\degg f>1$ is abelian and does not depend on the parameter $q$ (although its isomorphism class does), if we allow $0<\chara(\FF)\leq \degg f$ then we can obtain non-abelian automorphism groups.

\subsection{Conventions and notation}\label{SS:intro:notation}

Throughout the paper, $\FF$ will denote an arbitrary field, with multiplicative group denoted by $\FF^*$. The integers, nonnegative integers and positive integers will be denoted by $\ZZ$, $\ZZ_{\geq 0}$ and $\ZZ_{>0}$, respectively. Given a set $E$, the identity map on $E$ will be denoted by $1_E$. 

The relation $hx=xf(h)$ implies that $hx^2=x^2 f(f(h))$ and similarly for higher powers of $x$ and $y$. To deal with this type of commutation we introduce the unital algebra endomorphism $\sigma:\FF[h]\longrightarrow \FF[h]$ which maps $h$ to $f(h)$. Then $f(h)=\sigma(h)$, $f(f(h))=\sigma^2(h)$, etc. Thus, for example, $hx^k=x^k \sigma^k(h)$, for all $k\geq 0$.

\section{Basic properties and first results on quantum generalized Heisenberg algebras}\label{S:basic}

For the reader's convenience we collect in this short section some basic results and properties of quantum generalized Heisenberg algebras. 

\subsection{Some (anti)-isomorphisms}\label{SS:basic:autos}

Whenever possible, we will exploit the symmetry between $x$ and $y$ in the defining relations \eqref{E:intro:def:qgha}. This is materialized by the anti-automorphism of order 2, $\iota:\hqfg\longrightarrow\hqfg$, fixing $h$ and interchanging $x$ and $y$. Applying $\iota$ to an equation in $\hqfg$ will reverse the roles of $x$ and $y$ at the cost of inverting the order of multiplication. In this way we can show the equivalence between right and left versions of properties like being Noetherian or primitive.

The isomorphism below will be useful, e.g.\ in Proposition~\ref{P:noetherian}, to adjust the independent term of $f$ in $\hqfg$.

\begin{lemma}\label{L:basic:auto1}
For any $\alpha\in\FF$ we have the isomorphism $\hqfg\simeq\hqfg[f(h-\alpha)+\alpha, g(h-\alpha)]$.
\end{lemma}
\begin{proof}
Let $\phi:\FF\langle h, x, y\rangle\longrightarrow\FF\langle h, x, y\rangle$ be the automorphism of the fee algebra on $h, x, y$ defined on the generators by $\phi(x)=x$, $\phi(y)=y$ and $\phi(h)=h-\alpha$. Then $\phi(hx-xf(h))=(h-\alpha)x-xf(h-\alpha)=hx-x(f(h-\alpha)+\alpha)$; similarly, $\phi(yh-f(h)y)=yh-(f(h-\alpha)+\alpha)y$ and $\phi(yx-qxy-g(h))=yx-qxy-g(h-\alpha)$. Hence, $\phi$ maps the defining ideal of $\hqfg$ to the defining ideal of $\hqfg[f(h-\alpha)+\alpha, g(h-\alpha)]$ and induces the claimed isomorphism.
\end{proof}

\subsection{Basic structure and $\ZZ$-grading}\label{SS:basic:grading}

There is a natural $\ZZ$-grading
obtained by setting $x$ in degree $1$, $h$ in degree $0$ and $y$ in degree $-1$. It gives the decomposition
\begin{equation}\label{E:basic:weight_space_dec}
\hqfg=\bigoplus_{k\in\ZZ} \hqfg_k,
\end{equation}
where $\hqfg_k$ denotes the vector subspace of homogeneous elements of degree $k$.

\begin{prop}[{\cite{LR20arXiv}}]\label{P:basic}
Let $\hqfg$ be a quantum generalized Heisenberg algebra. Then the following hold.
\begin{enumerate}[label={(\alph*)}]
\item For any basis $\pb{v_j}_{j\in \ZZ_{\geq 0}}$ of $\FF[h]$, the set $\pb{x^i v_j y^k \mid i, j, k\in \ZZ_{\geq 0}}$ is a basis of $\hqfg$.
\item $\hqfg$ is a domain if and only if $q\neq 0$ and  $\degg f\geq 1$.
\item Concerning the decomposition \eqref{E:basic:weight_space_dec} we have, for $k\geq 0$,
\begin{equation*}
 \hqfg_0=\bigoplus_{i\geq 0}x^i \FF[h]y^i, \quad  \hqfg_k=x^k\, \hqfg_0 \quad \mbox{and}\quad \hqfg_{-k}=\hqfg_0\, y^{k}.
\end{equation*}
\item Suppose that $\degg f>1$ and denote the center of $\hqfg$ by $\Z(\hqfg)$. Then:
\begin{enumerate}[label={(\roman*)}]
\item $\hqfg_0$ is the centralizer of $h$.
\item If $q$ is not a root of unity, then $\Z(\hqfg)=\FF$.
\item If $q$ is a primitive $\ell$-th root of unity and $g(h)=\sigma(a)-qa$ for some $a\in\FF[h]$, then $\Z(\hqfg)=\FF[Z^\ell]$, where $Z=q(xy-a)$.
\end{enumerate}
\end{enumerate} 
\end{prop}

\subsection{A (non-standard) $\ZZ^2$-filtration}\label{SS:basic:filtration}

Let $0\neq a\in\hqfg$. Then, by Proposition~\ref{P:basic}, there exist unique elements $p_{ij}(h)\in\FF[h]$ so that
\begin{equation*}
a=\sum_{i, j\geq 0} x^i p_{ij}(h) y^j.
\end{equation*}
We consider the lexicographical order on the (finite) set $\operatorname{\sf supp}(a)=\pb{(i, j)\in\ZZ^2\mid p_{ij}(h)\neq 0}$
and define 
\begin{equation}\label{E:lexdeg}
\DEg a=\max\operatorname{\sf supp}(a),
\end{equation}
with the convention that $\DEg 0=(-\infty, -\infty)$. Then set 
\begin{align*}
\mathcal F_{(\alpha, \beta)}&=\pb{a\in\hqfg\mid \DEg a\leq (\alpha, \beta)} \quad \text{and}\quad
\mathcal F^-_{(\alpha, \beta)}&=\pb{a\in\hqfg\mid \DEg a< (\alpha, \beta)},
\end{align*}
for all $\alpha, \beta\geq 0$. 

\begin{lemma}\label{L:basic:filtration}
Let $\alpha, \tilde\alpha, \beta, \tilde\beta\in\ZZ_{\geq 0}$ and $p, \tilde p\in\FF[h]$. Then:
\begin{enumerate}[label=(\alph*)]
\item $\displaystyle(x^\alpha py^\beta)(x^{\tilde\alpha} \tilde py^{\tilde\beta})-q^{\tilde\alpha \beta}x^{\alpha+\tilde\alpha} \sigma^{\tilde\alpha}(p)\sigma^\beta(\tilde p)y^{\beta+\tilde\beta}\in{\mathcal F}^-_{(\alpha+\tilde\alpha,\beta+\tilde\beta)}$.

\item $\mathcal F_{(\alpha, \beta)}\mathcal F_{(\tilde\alpha, \tilde\beta)}\subseteq{\mathcal F}_{(\alpha+\tilde\alpha,\beta+\tilde\beta)}$ and $\seq{\mathcal F_{(i, j)}}_{i,j\geq 0}$ defines an increasing filtration of $\hqfg$. 

\item If $q\neq 0$ and $\degg f\geq 1$ then $\DEg ab=\DEg a+\DEg b$, for all $a, b\in\hqfg$.
\end{enumerate}
\end{lemma}
\begin{proof}
The first claim can be easily proved by induction on $\beta$. Whence, the remaining claims follow, noting that $\sigma$ is injective if and only if $\degg f\geq 1$.
\end{proof}

Let $\mathsf{Gr}\seq{\hqfg}=\bigoplus\limits_{\alpha, \beta\geq0} \sfrac{\mathcal F_{(\alpha, \beta)}}{\mathcal F^-_{(\alpha, \beta)}}$ be the associated graded algebra. Then $\mathsf{Gr}\seq{\hqfg}$ is generated by the images of the canonical generators $\overline x$, $\overline y$ and $\overline h$, with relations:
\begin{equation}\label{E:graded:qgha}
\overline h\overline x=\overline xf(\overline h), \quad \overline y\overline h=f(\overline h)\overline y, \quad \overline y\,\overline x=q\overline x\,\overline y.
\end{equation} 
In other words, we have $\mathsf{Gr}\seq{\hqfg}\simeq\hqfg[f,0]$.

\subsection{Gelfand-Kirillov dimension and relation to generalized down-up algebras}\label{SS:basic:GK}

It is known (\cite[Cor.\ 3.2]{BR98} and \cite[Cor.\ 2.4]{CS04}) that all (generalized) down-up algebras have Gelfand-Kirillov dimension $3$. In view of Proposition~\ref{P:intro:gduaRqgha}, if $\degg f\leq 1$ then $\hqfg$ is a generalized down-up algebra and thus $\operatorname{GKdim}\hqfg=3$. We will see next that this no longer holds if $\degg f>1$. (For the definition and properties of Gelfand-Kirillov dimension see \cite{KL00}.)

\begin{prop}\label{P:basic:GK}
Let $\hqfg$ be a qGHA. Then $\operatorname{GKdim}\hqfg=3$ if and only if $\degg f\leq 1$. If $\degg f>1$ then $\operatorname{GKdim}\hqfg\geq 4$.
\end{prop}
\begin{proof}
Assume that $\degg f>1$. Although \cite[Lem.\ 6.5]{KL00} is phrased in terms of $\ZZ$-filtered algebras, its proof carries through to general filtrations like the one we defined in Subsection~\ref{SS:basic:filtration}. Thus, by that result, it is enough to show that $\operatorname{GKdim}\mathsf{Gr}\seq{\hqfg}\geq 4$. Equivalently, we can assume without loss of generality that $g=0$.

Let $I=\seq{x}$ be the two-sided ideal of $\hqfg[f,0]$ generated by $x$. Since $x$ is right regular, by Proposition~\ref{P:basic}, then \cite[Prop.\ 3.15]{KL00} says that $\operatorname{GKdim}\hqfg[f,0]\geq \operatorname{GKdim}\hqfg[f,0]/I +1$. Thus, it is enough to show that $\operatorname{GKdim}\hqfg[f,0]/I \geq 3$. 

The algebra $\hqfg[f,0]/I$ is isomorphic to the unital subalgebra $R$ of $\hqfg[f,0]$ generated by $h$ and $y$. Then $R$ can be seen as the Ore extension $\FF[h][y;\sigma]$, with $\sigma(h)=f(h)$. Since $\degg f>1$, it is obvious that the only finite-dimensional subspaces of $\FF[h]$ which are $\sigma$-stable are $0$ and $\FF$. Thus, $\FF1\oplus\FF h$ is not contained in any finite-dimensional $\sigma$-stable subspace of $\FF[h]$. Then, by \cite[Thm.\ 1.1]{jZ97}, it follows that $\operatorname{GKdim}R=\operatorname{GKdim}\FF[h][y;\sigma]\geq \operatorname{GKdim}\FF[h]+2=3$. 
\end{proof}

Thanks to Proposition~\ref{P:intro:gduaRqgha}, we know that generalized down-up algebras are precisely the qGHA $\hqfg$ with $\degg f\leq1$. On the other hand, if $\degg f>1$ then Proposition~\ref{P:basic:GK} says that $\operatorname{GKdim}\hqfg\geq 4$, so in this case $\hqfg$ cannot be isomorphic to a generalized down-up algebra. So we obtain the following, which can be thought of as a prelude to the classification result in Theorem~\ref{T:iso}.

\begin{cor}\label{gduaimpdf1}
The quantum generalized Heisenberg algebra $\hqfg$ is isomorphic to a generalized down-up algebra if and only if $\degg f\leq1$.
\end{cor}

\section{Noetherian quantum generalized Heisenberg algebras and down-up algebras}\label{S:noetherian}

Any generalized down-up algebra has the property that it is Noetherian if and only if it is a domain (see~\cite{KMP99} and \cite{CS04}). It is natural to wonder whether this property still holds for a qGHA. In this section we determine when a quantum generalized Heisenberg algebra $\hqfg$ is Noetherian and give a negative answer to the above question.

\begin{prop}\label{P:noetherian}
A qGHA $\hqfg$ is right (or left) Noetherian if and only if $\degg f=1$ and $q\neq0$.
\end{prop}
\begin{proof}
Since $\hqfg$ is isomorphic to its opposite algebra (via the anti-automorphism $\iota$ defined in Subsection~\ref{SS:basic:autos}) it is enough to consider the Noetherian property on the left. If $\degg f=1$ and $q\neq0$ then, by Proposition~\ref{P:intro:gduaRqgha}, $\hqfg\simeq L(-g,a,q,-b)$ for $f(h)=ah+b$. Since $a,q\neq0$, \cite[Prop.\ 2.5]{CS04} proves that $\hqfg$ is Noetherian.
	
Let us now prove the converse. Suppose first that $\degg f\neq1$. Then $F(h)=f(h)-h$ is not a constant polynomial.

\smallskip

\textbf{Case 1.} Suppose there is $\beta\in \FF$ such that $F(\beta)=0$. Then $f(\beta)=\beta$. By Lemma~\ref{L:basic:auto1} with $\alpha=-\beta$ we have $\hqfg\simeq\mathcal H_q(\tilde{f},g(h+\beta))$, where $\tilde{f}(h)=f(h+\beta)-\beta$. Notice that $\tilde{f}(0)=f(\beta)-\beta=0$, so $\tilde{f}\in h\FF[h]$. Thus, since $\degg \tilde{f}=\degg f$ and being Noetherian is invariant under isomorphism, we can assume without loss of generality that $f(h)\in h\FF[h]$.
	
Consider the chain of left ideals
\begin{equation}\label{E:Noetherian:chain}
I_0\subseteq \cdots\subseteq I_n\subseteq I_{n+1}\subseteq \cdots
\end{equation}	
where $I_n=\sum\limits_{i=0}^{n} \hqfg hy^i$.
As we know from Proposition~\ref{P:basic} that $\hqfg=\oplus_{j,k\ge 0}x^j\FF[h]y^k$, we can write
	\begin{align}
	I_n=\sum\limits_{j,k}^{}\sum\limits_{i=0}^{n}x^j\FF[h]y^khy^i=\label{yy} \sum\limits_{j,k}^{}\sum\limits_{i=0}^{n}x^j\FF[h]\sigma^{k}(h) y^{i+k}.
	\end{align}

We will show that, for all $n\geq 0$, the inclusion $I_n\subseteq I_{n+1}$ is strict. Otherwise, $hy^{n+1}\in I_{n}$. Then it follows that 
\begin{equation}\label{Equ32}
hy^{n+1}\in\sum\limits_{j,k}^{}\sum\limits_{i=0}^{n}x^j\FF[h]\sigma^{k}(h) y^{i+k}
\end{equation}
and, again by Proposition~\ref{P:basic}, we can take $j=0$ and $k=n-i+1$ in \eqref{Equ32}. Hence, there exist $p_i(h)\in \FF[h]$ such that 
\begin{equation*}
hy^{n+1}=\left(  \sum\limits_{i=0}^{n}p_i(h)\sigma^{n+1-i}(h)\right) y^{n+1},
\end{equation*}
and thus
\begin{equation}\label{Equ17}
h=\sum\limits_{i=0}^{n}p_i(h)\sigma^{n+1-i}(h). 
\end{equation}

We will show that $\sigma^k(h)\in f\FF[h]$ for all $k\ge 1$. The $k=1$ case is just the definition $\sigma(h)=f$. Recall that $f(h)\in h\FF[h]$, hence there exists $\zeta(h)\in \FF[h]$ such that $f(h)=h\zeta(h)$. For the inductive step, assuming that $\sigma^k(h)=f(h)p(h)$ for some $p(h)\in\FF[h]$, we have
	\begin{align*}
	\sigma^{k+1}(h)=\sigma(f(h)p(h))=\sigma(h\zeta(h)p(h))=\sigma(h)\sigma(\zeta(h)p(h))=f(h)\sigma(\zeta(h)p(h))\in f\FF[h].
	\end{align*}
Then, in particular, equation \eqref{Equ17} implies that $h\in f\FF[h]$. Since we are also assuming that $f\in h\FF[h]$, it follows that $\degg f=1$, which is a contradiction. This proves that \eqref{E:Noetherian:chain} is a strict ascending chain of left ideals and $\hqfg$ is not left Noetherian.
	
\smallskip

\textbf{Case 2.}  Suppose now that there is no $\beta\in \FF$ such that $F(\beta)=0$. Since $F$ is not a constant polynomial, there is a finite field extension $\mathbb E$ of $\FF$ and $\beta\in \mathbb E$ such that $F(\beta)=0$. Then we consider the algebra $\hqfg\otimes_{\FF}\mathbb E$. Since $\mathbb E$ is a finite extension of $\FF$ it follows that $\hqfg\otimes_{\FF}\mathbb E$ is a finite module over $\hqfg$. If $\hqfg$ were Noetherian, we would conclude that $\hqfg\otimes_{\FF}\mathbb E$ is also Noetherian. But $\hqfg\otimes_{\FF}\mathbb E$ is just a quantum generalized Heisenberg algebra defined over the field $\mathbb E$ with exactly the same parameters $q$, $f$ and $g$. As there is $\beta\in\mathbb E$ with $f(\beta)=\beta$, Case~1 implies that $\hqfg\otimes_{\FF}\mathbb E$ is not Noetherian; thus, neither is $\hqfg$.

\smallskip
	
It remains to consider $q=0$ with $\degg f=1$. But in this case, by Proposition~\ref{P:intro:gduaRqgha}, $\hqfg$ is isomorphic to a generalized down-up algebra $L(-g, a,0,-b)$ and by~\cite[Prop.\ 2.6]{CS04} $\hqfg$ is not Noetherian.
\end{proof}

\begin{remark}
The above result also implies Corollary~\ref{gduaimpdf1} under the additional assumption that $q\neq 0$, using the fact that only the Noetherian generalized down-up algebras are domains (\cite{CS04}) and that $\hqfg$ is a domain if $\degg f>0$ and $q\neq 0$.
\end{remark}

\section{Classification of quantum generalized Heisenberg algebras}\label{S:class}

It is quite common for different sets of generators and relations to yield the same intrinsic structure. To detect this, we need to study all possible isomorphisms among qGHA. Isomorphisms can also be a very powerful way of simplifying arguments and computations. The isomorphism problem for generalized Heisenberg algebras $\hh(f)$ over the field of complex numbers was tackled in~\cite{LZ15}. Here we consider the classification problem for quantum generalized Heisenberg algebras $\hqfg$ over an arbitrary field.

By Corollary~\ref{gduaimpdf1}, we know that a qGHA with $\degg f\leq1$ cannot be isomorphic to another qGHA with $\degg f>1$. Moreover, in case $\degg f\leq1$ we obtain a generalized down-up algebra, whose isomorphisms have been studied elsewhere (see~\cite{BR98}, \cite{CM00}, \cite{CL09} and \cite{LZ15}). We will thus focus on the case $\degg f>1$. Moreover, to avoid technicalities, throughout the remainder of this paper we will assume that $q\neq 0$.

It will be proved in this section that the isomorphism relation among the quantum generalized Heisenberg algebras with $\degg f>1$ can be phrased in very concrete geometric terms, very much like in \cite[Thm.\ 3.28]{BJ01} (compare Proposition~\ref{Prop7}). It will follow in particular that, in case $q\neq 0$ and $\degg f>1$, the parameter $q$, as well as the integers $\degg f$ and $\degg g$, are invariant under isomorphism, showing that qGHA are indeed a vast generalization of generalized Heisenberg algebras and generalized down-up algebras.

We begin by listing three types of isomorphisms from which, as we will see, all other isomorphisms can be determined.

\begin{prop}\label{Prop7} 
Let $q\in\FF$ and $f,g \in\FF[h]$. The following define isomorphisms of qGHA.
\begin{enumerate}[label=\Roman*.,ref=\Roman*]
\item For all $\alpha\in\FF$, $\tau_{\alpha}:\hqfg\longrightarrow \hqfg[f(h-\alpha)+\alpha, g(h-\alpha)]$, defined on the canonical generators by $x\mapsto  x$, $y\mapsto  y$ and $h\mapsto  h-\alpha$.\label{P:7:I}

\item For all $\lambda\in\FF^*$, $\sigma_{\lambda}:\hqfg\longrightarrow \hqfg[\lambda f(\lambda^{-1}h), g(\lambda^{-1}h)]$, defined on the canonical generators by $x\mapsto  x$, $y\mapsto  y$ and $h\mapsto  \lambda^{-1}h$.\label{P:7:II}

\item For all $\lambda, \mu\in\FF^*$, $\rho_{\lambda,\mu}:\hqfg\longrightarrow \hqfg[f,\lambda\mu g]$, defined on the canonical generators by $x\mapsto  \lambda^{-1}x$, $y\mapsto  \mu^{-1}y$ and $h\mapsto h$.\label{P:7:III}
\end{enumerate}
\end{prop}
\begin{proof}
The isomorphism in~\ref{P:7:I} is just the one from Lemma~\ref{L:basic:auto1}. The remaining ones can be easily checked just as in the proof of that result.
\end{proof}

We will refer to the isomorphisms from Proposition~\ref{Prop7} either by the notation established there, if we need to specify the parameters involved (e.g.\ $\rho_{\lambda,\mu}$), or by a \textit{type}, matching the numbering above (e.g., an isomorphism of Type~\ref{P:7:III} is one of the form $\rho_{\lambda,\mu}$, for some $\lambda, \mu\in\FF^*$).

\begin{thm}\label{T:iso}
Assume $q\neq0$ and $\degg f>1$. Then $\hqfg\simeq \hqqfg$ if and only if $q=q'$ and $(f',g')$ is obtained from $(f,g)$ via transformations of types \ref{P:7:I}, \ref{P:7:II}, \ref{P:7:III} defined in Proposition~\ref{Prop7}. It follows in particular that $\degg f=\degg f'$ and $\degg g=\degg g'$.
\end{thm}
\begin{proof}
If $q=q'$ and $(f',g')$ is obtained from $(f,g)$ by the transformations defined in Proposition~\ref{Prop7}, then clearly $\hqfg\simeq \hqqfg$.
	
For the converse statement, suppose that $q\neq0$, $\degg f>1$ and $\hqfg\simeq \hqqfg$. Then by Propositions~\ref{P:basic} and~\ref{P:noetherian}, $\hqfg$ is a non-Noetherian domain, so the same holds for $\hqqfg$ and thus $q'\neq0$ and $\degg f'>1$. 
	
Let $\varphi:\hqfg\longrightarrow\hqqfg$ be an isomorphism. To avoid any ambiguity, we use $x$, $y$ and $h$ for the canonical generators of $\hqfg$ and $X$, $Y$ and $H$ for the canonical generators of $\hqqfg$. The automorphism of $\FF[H]$ which sends $H$ to $f'(H)$ is denoted $\sigma'$.
	
From the application of $\varphi$ to the first defining relation in~\eqref{E:intro:def:qgha} we obtain $\varphi(h)\varphi(x)=\varphi(x)\varphi(f(h))=\varphi(x)f(\varphi(h))$. Then, taking the lexicographical degree defined in~\eqref{E:lexdeg} we get $\DEg\varphi(h)+\DEg\varphi(x)=\DEg\varphi(x)+\DEg f(\varphi(h))$. As $\varphi(x)\neq 0$ and $\DEg f(\varphi(h))=\degg f\cdot\DEg\varphi(h)$, we infer that $\DEg\varphi(h)=\degg f\cdot\DEg\varphi(h)$. Finally, since $\degg f>1$ and $\varphi(h)\neq 0$, it must be that $\DEg \varphi(h)=(0,0)$. Hence, $\varphi(h)\in\FF[H]$. The same argument with $\varphi^{-1}$ shows that $\varphi^{-1}(H)\in\FF[h]$ and thence $\varphi_{|_{\FF[h]}}:\FF[h]\longrightarrow\FF[H]$ is an isomorphism. It follows that $\varphi(h)=aH+b$ for some $a,b\in\FF$ with $a\neq0$.
	
The composition $\tau_{ba^{-1}}\circ\varphi$ gives an isomorphism $\hqfg\longrightarrow\hqqfg[f'',g'']$, where $(f'',g'')$ is obtained from $(f',g')$ by a transformation of Type~\ref{P:7:I} and $\tau_{ba^{-1}}\circ\varphi(h)=\tau_{ba^{-1}}(aH+b)=aH$. So there is no loss in generality in assuming that $\varphi(h)=aH$. Similarly, using the transformation $\sigma_a$ of Type~\ref{P:7:II}, we can assume further that $\varphi(h)=H$.

Let $\mathsf{C}_{A}(a)$ denote the centralizer of an element $a\in A$, where $A$ is an algebra. By Proposition~\ref{P:basic} we have
\begin{equation*}
\varphi(\mathsf{C}_{\hqfg}(h))=\mathsf{C}_{\hqqfg}(\varphi(h)) =\mathsf{C}_{\hqqfg}(H)=\hqqfg_0 =\bigoplus_{i\ge0}X^i\FF[H]Y^i.
\end{equation*}
In particular, $\varphi(xy)\in \bigoplus_{i\ge0}X^i\FF[H]Y^i$, whence we have $\DEg \varphi(xy)=(i,i)$ for some $i\ge 1$ ($i=0$ would contradict the injectivity of $\varphi$). By symmetry, $\DEg \varphi^{-1}(XY)=(j,j)$, for some $j\ge 1$. Our immediate goal is to show that $i=1=j$.

Write
\begin{equation}\label{Equ36}
\varphi^{-1}(XY)=\sum_{k=0}^{j}x^kp_k(h)y^k, \qquad \text{with $p_j\neq 0$.}  
\end{equation}
Applying $\varphi$ to both sides of~\eqref{Equ36} and computing the lexicographical degree we obtain
\begin{align*}
(1,1)=\DEg XY=j\DEg\varphi(x)+\DEg\varphi(p_j(h))+j\DEg\varphi(y)=j\DEg\varphi(xy)=(ij,ij).
\end{align*}
So indeed $i=1=j$.
	
Since $\DEg\varphi(x), \DEg\varphi(y)\neq(0,0)$, there are just two cases to consider, although we will show that the second one can never hold.

\smallskip

\textbf{Case 1.} $\DEg\varphi(x)=(1,0)$ and $\DEg\varphi(y)=(0,1)$. So we can write
\begin{equation*}
\varphi(x)=Xp_1(H)+\sum\limits_{k\ge0}^{}\alpha_k(H)Y^k  \quad \text{and} \quad  \varphi(y)=p_2(H)Y+\beta(H),
\end{equation*}
for some $p_1,p_2, \alpha_k, \beta\in\FF[H]$ with $k\geq 0$ and $p_1,p_2\neq 0$.
	
Exploiting the relation $H\varphi(x)=\varphi(h)\varphi(x)=\varphi(hx)=\varphi(xf(h))=\varphi(x)f(H)$, we get
\begin{equation*}
HXp_1(H)+\sum\limits_{k\ge0}^{}H\alpha_k(H)Y^k=Xp_1(H)f(H)+\sum\limits_{k\ge0}^{}\alpha_k(H)Y^kf(H).
\end{equation*}
Writing the above in the basis given in Proposition~\ref{P:basic} and equating corresponding terms yields $f'(H)p_1(H)=p_1(H)f(H)$ and $H\alpha_k(H)=\alpha_k(H)(\sigma')^k(f(H))$, for all $k\ge0$. In particular, as $p_1\neq 0$, we deduce that $f'=f$. Also, $\alpha_k\neq0$ for some $k$ implies that $H=(\sigma')^k(f(H))$, which is impossible since $(\sigma')^k(f(H))=f((\sigma')^k(H))$ has degree strictly larger than 1 for any $k\geq 0$. Hence, $\alpha_k=0$ for all $k$. Similarly, using $yh=f(h)y$, we can deduce that $\beta(H)=0$. 

Therefore we have
\begin{equation*}
\varphi(x)=Xp_1(H) \quad \text{and}\quad \varphi(y)=p_2(H)Y. 
\end{equation*}
Now we substitute these expressions into the relation $\varphi(y)\varphi(x)-q\varphi(x)\varphi(y)=g(H)$ and obtain
\begin{align}\label{E:ggprime}
g(H)&=p_2(H)YXp_1(H)-qXp_1(H)p_2(H)Y\notag\\&=q'p_2(H)XYp_1(H)+p_2(H)g'(H)p_1(H)-qXp_1(H)p_2(H)Y\\&=X(q'p_2(f(H))p_1(f(H))-qp_1(H)p_2(H))Y+p_2(H)g'(H)p_1(H).\notag
\end{align}
Comparing coefficients again we get
\begin{equation*}
q'p_2(f(H))p_1(f(H))=qp_1(H)p_2(H).
\end{equation*}
Since $\degg f>1$, this is possible only if $p_1(H),p_2(H)\in\FF^*$, say $p_1(H)=\lambda^{-1}$ and $p_2(H)=\mu^{-1}$, in which case it implies that $q=q'$. Finally, \eqref{E:ggprime} also gives $g(H)=p_2(H)g'(H)p_1(H)=\lambda^{-1}\mu^{-1}g'(H)$. Thus, $\varphi=\rho_{\lambda, \mu}$ is a transformation of Type~\ref{P:7:III}.
	
\smallskip

\textbf{Case 2.}  $\DEg\varphi(x)=(0,1)$ and $\DEg\varphi(y)=(1,0)$. In this case, we can write
\begin{equation*}
\varphi(x)=p(H)Y+\beta(H),
\end{equation*}
for some $p, \beta\in\FF[H]$ with $p\neq 0$.
Using the relation $H\varphi(x)=\varphi(x)f(H)$, we get 
\begin{equation}\label{E:isos:case2}
Hp(H)Y+H\beta(H)=p(H)Yf(H)+\beta(H)f(H).
\end{equation}
Writing~\eqref{E:isos:case2} in normal form and comparing coefficients, we deduce in particular that 
\begin{equation*}
Hp(H)=p(H)\sigma'(f(H))=p(H)f(f'(H)).
\end{equation*}
So $H=f(f'(H))$, which contradicts the fact that $\degg f>1$.
\end{proof}

\section{Automorphisms of quantum generalized Heisenberg algebras}\label{S:automorphisms}

Automorphisms reflect the inner symmetries of an algebra and are thus an extremely useful tool for understanding it intrinsically. 
Having classified the qGHA by isomorphism, it is thence natural to turn to their automorphism groups. The automorphisms groups of generalized Heisenberg algebras were determined in \cite{sL17}. If $\degg f=1$ and $q\neq 0$ then $\hqfg$ is a Noetherian generalized down-up algebra, and its automorphisms have been studied in \cite{CL09}. Therefore, we will continue to assume that $\degg f>1$. 

For each $\lambda\in\FF^*$, let $\phi_{\lambda}$ be the automorphism of $\hqfg$ defined by $\phi_{\lambda}(x)=\lambda x$, $\phi_{\lambda}(y)=\lambda^{-1}y$ and $\phi_{\lambda}(h)=h$. Notice that, in terms of the notation from Proposition~\ref{Prop7}, $\phi_{\lambda}=\rho_{\lambda^{-1}, \lambda}$.

\begin{prop}\label{Prop9}
Assume that $q\neq 0$ and $\degg f>1$. The following hold.
\begin{enumerate}[label=(\alph*)]
\item Any automorphism of $\hqfg$ restricts to an automorphism of $\FF[h]$, and $x$ and $y$ are eigenvectors. \label{Prop9a}
\item If $g\neq 0$ then $\pb{\phi\in \autt{\hqfg}\mid\phi(h)=h}=\pb{\phi_{\lambda}|\lambda\in\FF^* }\simeq\FF^*$, and this is a central subgroup of $\autt{\hqfg}$. \label{Prop9b}
\item If $g\neq 0$ and either $\chara(\FF)=0$ or $\chara(\FF)>\degg f$ then $\pb{\phi\in\autt{\hqfg}\mid\phi(x)=x }$ is a finite cyclic subgroup whose order divides $(\degg f)-1$.\label{Prop9c}
\end{enumerate}
\end{prop}
\begin{proof}
Let $\phi$ be an automorphism of $\hqfg$. Since $q\neq 0$ and $\degg f>1$, the proof of Theorem~\ref{T:iso} applies and shows that 
$\phi$ restricts to an automorphism of $\FF[h]$, say $\phi(h)=ah+b$ for some $a,b\in\FF$ with $a\neq0$. Consider the grading defined in~\eqref{E:basic:weight_space_dec}. For $k\geq0$, let $\xi_k\in\hqfg_k=x^k\hqfg_0$, say $\xi_k=x^k\theta$ with $\theta\in\hqfg_0$. We have $p(h)\xi_k=p(h)x^k\theta=x^k\sigma^k(p(h))\theta=\xi_k\sigma^k(p(h))$, for all $p(h)\in\FF[h]$. Similarly, for $k\leq 0$, $\xi_kp(h)=\sigma^{-k}(p(h))\xi_k$.

Write $\phi(x)=\sum_{k\in\ZZ}\xi_k$, with $\xi_k\in\hqfg_k$. Then 
\begin{equation}\label{E:autos:prelim:1}
\phi(h)\phi(x)= \sum_{k\in\ZZ}\phi(h)\xi_k=\sum_{k\geq 0}\xi_k\sigma^k(\phi(h))+\sum_{k<0}\phi(h)\xi_k.
\end{equation}
On the other hand, 
\begin{equation}\label{E:autos:prelim:2}
\phi(h)\phi(x)= \phi(hx)=\phi(x)\phi(\sigma(h))=\sum_{k\geq 0}\xi_k\phi(\sigma(h))+\sum_{k<0}\sigma^{-k}(\phi(\sigma(h)))\xi_k.
\end{equation}
Equating homogeneous terms of the same degree in \eqref{E:autos:prelim:1} and \eqref{E:autos:prelim:2}, and using the fact that $\hqfg$ is a domain, we deduce that $\sigma^k(\phi(h))=\phi(\sigma(h))$ for all $k\geq 0$ such that $\xi_k\neq 0$ and $\phi(h)=\sigma^{-k}(\phi(\sigma(h)))$ for all $k< 0$ such that $\xi_k\neq 0$. Note that $\degg\phi(h)=1$, $\degg\phi(\sigma(h))=\degg f$, $\degg \sigma^k(\phi(h))=(\degg f)^k$ and $\degg \sigma^{-k}(\phi(\sigma(h)))=(\degg f)^{1-k}$. So since $\degg f>1$, the only possibility is that $\phi(x)$ is homogeneous of degree $1$. Similarly, $\phi(y)$ is homogeneous of degree $-1$. This shows that the automorphism $\phi$ is a homogeneous map with respect to the grading~\eqref{E:basic:weight_space_dec}.

Thus, there exist $\theta_x, \theta_y\in\hqfg_0$ such that $\phi(x)=x\theta_x$ and $\phi(y)=\theta_y y$. Applying the same reasoning to $\phi^{-1}$ and noting that the group of units of $\hqfg$ is $\FF^*$, we conclude that $\theta_x, \theta_y\in\FF^*$, which proves~\ref{Prop9a}.
	
For~\ref{Prop9b}, assume that $\phi(h)=h$ and $g\neq 0$. By \ref{Prop9a} there exist $\lambda,\mu\in\FF^*$ such that $\phi(x)=\lambda x$ and $\phi(y)=\mu y$. Then, applying $\phi$ to one of the defining relations, we get
\begin{align*}
\lambda\mu g(h)=\lambda\mu(yx-qxy)=\phi(yx-qxy)=\phi(g(h))=g(h),
\end{align*}
which yields $\lambda\mu=1$ and hence $\phi=\phi_{\lambda}$. This proves the equality in \ref{Prop9b}, and the isomorphism $\{ \phi_{\lambda}\mid\lambda\in\FF^* \}\cong\FF^*$ is clear as $\phi_{\lambda}\circ\phi_{\mu}=\phi_{\lambda\mu}$ for all $\lambda,\mu\in\FF^*$.
	
Next we show that the subgroup $\{\phi_{\lambda}\mid \lambda\in\FF^*  \}$ is central in $\autt{\hqfg}$. Let $\lambda\in\FF^*$, and suppose $\psi\in\autt{\hqfg}$ is arbitrary. By \ref{Prop9a} we know that $\psi(h)\in\FF[h]$, thus $\phi_{\lambda}\circ\psi(h)=\psi\circ\phi_{\lambda}(h)$. As $x$ and $y$ are eigenvectors for $\psi$, we see that $\phi_{\lambda}\circ\psi$ and $\psi\circ\phi_{\lambda}$ also agree on these generators, so we can conclude that $\phi_{\lambda}\circ\psi=\psi\circ\phi_{\lambda}$.
	
To prove part \ref{Prop9c}, suppose that $\phi\in\autt{\hqfg}$ and $\phi(x)=x$. We know already that $\phi(h)=ah+b$ and $\phi(y)=cy$, for some $a,b,c\in\FF$ with $a,c\neq0$. Then $xf(ah+b)=xf(\phi(h))=\phi(h)x=(ah+b)x=x(af(h)+b)$, and we obtain
\begin{equation}\label{Equ39}
f(ah+b)=af(h)+b.
\end{equation}
On the other hand, we have
\begin{align*}
cg(h)=c(yx-qxy)=\phi(yx-qxy)=\phi(g(h))=g(ah+b).
\end{align*}
Since $g\neq 0$, it follows that $c=a^{\degg g}$.

Now write $f(h)=\sum\limits_{k=0}^{n}a_kh^k$, where $n=\degg f$ and all $a_k\in\FF$. Applying the derivation operator $\frac{d}{dh}$ to (\ref{Equ39}) $n-1$ times yields $a^{n-1}f^{(n-1)}(ah+b)=af^{(n-1)}(h)$. As $f^{(n-1)}(h)=(n-1)!(na_nh+a_{n-1})$ and using the hypothesis on $\chara(\FF)$, we obtain
\begin{equation*}
na_na^nh+na^{n-1}a_nb+a^{n-1}a_{n-1}=ana_nh+aa_{n-1}
\end{equation*}
and conclude that  
\begin{equation}\label{Equ40}
a^{n-1}=1, \quad \text{and}\quad b=\frac{(a-1)a_{n-1}}{na_n}.
\end{equation}
In particular, since $f$ is fixed, $b$ is determined by $a$.

Let $U_{n-1}=\{\xi\in \FF^*\mid \xi^{n-1}=1  \}$. Then $U_{n-1}$ is a cyclic group whose order divides $n-1$. Define a map $\Psi:\{\phi\in \autt{\hqfg}\mid \phi(x)=x \}\longrightarrow U_{n-1}$ by $\Psi(\phi)=a$,
where $\phi(h)=ah+b$. Then $\Psi$ is well defined by \eqref{Equ40} and it is a group homomorphism. If $\Psi(\phi)=1$ for some $\phi\in\autt{\hqfg}$ with $\phi(x)=x$, then the above shows that $\phi(y)=a^{\degg g}y=y$ and $\phi(h)=h+b=h$. So $\phi$ is the identity on $\hqfg$. This shows that $\Psi$ is an injective group homomorphism and thus $\{\phi\in \autt{\hqfg}\mid \phi(x)=x \}$ is isomorphic to a subgroup of $U_{n-1};$ hence it is a finite cyclic group whose order divides $n-1$.
\end{proof}

We can now describe the structure of the automorphism group of $\hqfg$. We will treat the case $g=0$ separately as this case has a slightly different description.

\begin{thm}\label{Thm2}
Assume that $q\neq0$, $\degg f>1$ and $g\neq0$. Then we have the internal direct product decomposition
\begin{equation}\label{Equ41}
\autt{\hqfg}=\pb{\phi_{\lambda}\mid\lambda\in\FF^*} \dot{\times}\pb{\phi\in\autt{\hqfg}\mid\phi(x)=x}
\end{equation}
and $\pb{\phi_\lambda\mid \lambda\in\FF^*}\cong\FF^*$ is central in $\autt{\hqfg}$.

Moreover, if either $\chara(\FF)=0$ or $\chara(\FF)>\degg f$ then $\pb{\phi\in\autt{\hqfg} \mid \phi(x)=x}$ is a finite cyclic group whose order divides $(\degg f)-1$ and thus in this case $\autt{\hqfg}$ is an abelian group.
\end{thm}
\begin{proof}
Let $A=\{\phi_\lambda\mid \lambda\in\FF^*\}$ and $B=\{\phi\in\autt{\hqfg} \mid \phi(x)=x\}$. In order to prove the direct product decomposition \eqref{Equ41} we need to show that $\autt{\hqfg}=AB$, $A\cap B=\{1\}$ and that $A$, $B$ are normal subgroups of $\autt{\hqfg}$.
	
Let $\phi$ be an automorphism of $\hqfg$. By Proposition~\ref{Prop9}, there exist $a\in\FF^*$, $b\in\FF$ and $\lambda,\mu\in\FF^*$ such that $\phi(h)=ah+b$, $\phi(x)=\lambda x$ and $\phi(y)=\mu y$. Then $\phi_{\lambda^{-1}}\circ\phi(x)=x$. Write $\psi=\phi_{\lambda^{-1}}\circ\phi \in B$. Then we have $\phi_\lambda\psi=\phi$, so $\autt{\hqfg}=AB$. We have already seen in Proposition~\ref{Prop9} that $A$ is central in $\hqfg$; given that $\autt{\hqfg}=AB$, that implies that $B$ is also normal in $\hqfg$. Finally, the only automorphism in $A$ which fixes $x$ is $\phi_1=1_{\hqfg}$, by definition of $\phi_\lambda$, so the two subgroups have trivial intersection.
	
The last statement follows by Proposition~\ref{Prop9}\ref{Prop9c}, as $\autt{\hqfg}$ is then the product of the abelian subgroups $A$ and $B$.
\end{proof}

\begin{remark}\label{R:autos:param:xfix}
By the proof of Proposition~\ref{Prop9}, if $q\neq 0$ and $\degg f>1$ then, as a set, $\pb{\phi\in\autt{\hqfg} \mid \phi(x)=x}$ can be identified with 
\begin{equation*}
\pb{ (a, b)\in\FF^*\times\FF \mid f(ah+b)=af(h)+b \ \text{and} \ g(ah+b)=a^{\degg g}g(h)}, 
\end{equation*}
via the map $\phi\mapsto (a, b)$, where $\phi(h)=ah+b$. This mapping is well defined and one-to-one because we have shown that in this situation $\phi(y)$ depends only on $a$. Conversely, given $(a, b)\in\FF^*\times\FF$ such that $f(ah+b)=af(h)+b$ and $g(ah+b)=a^{\degg g}g(h)$, it is easy to see that there is an automorphism of $\hqfg$ sending $h$ to $ah+b$, $x$ to $x$ and $y$ to $a^{\degg g} y$. This observation reduces the computation of the subgroup $\pb{\phi\in\autt{\hqfg} \mid \phi(x)=x}$ to an arithmetical question involving just the ground field $\FF$ and the polynomials $f$ and $g$, but not the parameter $q$.
\end{remark}

Now we settle the case $g=0$.

\begin{thm}\label{Thm3}
Assume that $q\neq0$, $\degg f>1$ and $g=0$. Then we have the internal direct product decomposition
\begin{equation*}
\autt{\hqfg[f,0]}=\pb{\phi\in\autt{\hqfg[f,0]}\mid \phi(x)=x, \phi(y)=y} \dot{\times} \pb{\phi\in\autt{\hqfg[f,0]} \mid \phi(h)=h}
\end{equation*}
and $\pb{\phi\in\autt{\hqfg[f,0]} \mid \phi(h)=h}\simeq \FF^*\times\FF^*$.

If either $\chara(\FF)=0$ or $\chara(\FF)>\degg f$ then $\pb{\phi\in\autt{\hqfg} \mid \phi(x)=x, \phi(y)=y}$ is a finite cyclic group whose order divides $(\degg f)-1$ and thus in this case $\autt{\hqfg[f,0]}$ is an abelian group.
\end{thm}
\begin{proof}
Let $A=\{\psi\in\autt{\hqfg[f,0]} \mid \psi(h)=h \} $ and $B=\{\phi\in\autt{\hqfg[f,0]}\mid \phi(x)=x, \phi(y)=y \}$. We will describe $A$ first. 
Given $\psi\in A$, we know by Proposition~\ref{Prop9} that there exist $\lambda,\mu\in\FF^*$ so that $\psi(x)=\lambda x$ and $\psi(y)=\mu y$. Denote such an automorphism by $\psi_{\lambda,\mu}$ and recall that $\psi_{\lambda,\mu}(h)=h$ because $\psi_{\lambda,\mu}\in A$. Conversely, thanks to the defining relation $yx-qxy=0$, it is immediate to see that $\psi_{\lambda,\mu}$ is a well-defined automorphism of $\hqfg[f,0]$, for arbitrary $\lambda,\mu\in\FF^*$. Hence, $A=\pb{\psi_{\lambda,\mu}\mid \lambda,\mu\in\FF^*}$ and $A\simeq \FF^*\times\FF^*$, as groups.

Next, we prove that $A$ is central. Let $\phi\in\autt{\hqfg[f,0]}$ and $\lambda,\mu\in\FF^*$. Then $\phi\circ\psi_{\lambda, \mu}(h)=\phi(h)=\psi_{\lambda, \mu}\circ\phi (h)$ because $\psi_{\lambda, \mu}$ is the identity on $\FF[h]$. Since $x, y$ are eigenvectors for both $\phi$ and $\psi_{\lambda, \mu}$, it also follows that $\phi\circ\psi_{\lambda, \mu}$ and $\psi_{\lambda, \mu}\circ\phi$ agree on the generators $x, y$ and thus $\phi$ and $\psi_{\lambda, \mu}$ commute. The direct product decomposition then follows just as in the proof of Theorem~\ref{Thm2}.

Finally, as in Remark~\ref{R:autos:param:xfix}, the subgroup $B$ of automorphisms of $\hqfg[f,0]$ which fix $x$ and $y$ can be identified with the set $\pb{ (a, b)\in\FF^*\times\FF \mid f(ah+b)=af(h)+b}$. In case $\chara(\FF)=0$ or $\chara(\FF)>\degg f$, the same methods used in the proof of Proposition~\ref{Prop9}\ref{Prop9c} show that $B$ is cyclic and its order divides $(\degg f)-1$.
\end{proof}

The following example will show that, without the additional hypothesis on the characteristic of the base field $\FF$ imposed in Theorems~\ref{Thm2} and~\ref{Thm3}, the group $\autt{\hqfg}$ may be non-abelian.

\begin{exam} 
Suppose that $\chara(\FF)=p>2$, set $f(h)=h^p$ and let $g(h)=0$ or $g(h)=h^p-h$. Then the following define automorphisms of $\hqfg$:
\begin{eqnarray}
\phi(h)=h+1,\qquad \phi(x)=x,\qquad \phi(y)=y;\\
\psi(h)=2h,\qquad \psi(x)=x,\qquad \psi(y)=2y. 
\end{eqnarray}
This is because $f(h+1)=(h+1)^p=h^p+1=f(h)+1$, $f(2h)=2^ph^p=2h^p=2f(h)$, $g(h+1)=(h+1)^p-(h+1)=g(h)$ and $g(2h)=2^ph^p-2h=2^pg(h)$. However, $\phi\circ\psi(h)=2h+2\neq 2h+1=\psi\circ\phi(h)$, so the group $\autt{\hqfg}$ is not abelian. 
\end{exam}

\newcommand{\germ}{\mathfrak}

\bibliographystyle{plain}

\begin{thebibliography}{10}

\bibitem{BCG18}
F.~Bagarello, E.~M.~F. Curado, and J.~P. Gazeau.
\newblock Generalized {H}eisenberg algebra and (non linear) pseudo-bosons.
\newblock {\em J. Phys. A}, 51(15):155201, 16, 2018.

\bibitem{BJ01}
V.~V. Bavula and D.~A. Jordan.
\newblock Isomorphism problems and groups of automorphisms for generalized
  {W}eyl algebras.
\newblock {\em Trans. Amer. Math. Soc.}, 353(2):769--794, 2001.

\bibitem{BR98}
G.~Benkart and T.~Roby.
\newblock Down-up algebras.
\newblock {\em J. Algebra}, 209(1):305--344, 1998.

\bibitem{BBH11}
K.~Berrada, M.~El Baz, and Y.~Hassouni.
\newblock Generalized {H}eisenberg algebra coherent states for power-law
  potentials.
\newblock {\em Physics Letters A}, 375(3):298 -- 302, 2011.

\bibitem{CL09}
Paula A. A.~B. Carvalho and Samuel~A. Lopes.
\newblock Automorphisms of generalized down-up algebras.
\newblock {\em Comm. Algebra}, 37(5):1622--1646, 2009.

\bibitem{CM00}
Paula A. A.~B. Carvalho and Ian~M. Musson.
\newblock Down-up algebras and their representation theory.
\newblock {\em J. Algebra}, 228(1):286--310, 2000.

\bibitem{CS04}
Thomas Cassidy and Brad Shelton.
\newblock Basic properties of generalized down-up algebras.
\newblock {\em J. Algebra}, 279(1):402--421, 2004.

\bibitem{CRM01}
E~M~F Curado and M~A Rego-Monteiro.
\newblock Multi-parametric deformed {H}eisenberg algebras: a route to
  complexity.
\newblock {\em Journal of Physics A: Mathematical and General}, 34(15):3253,
  2001.

\bibitem{CRMR13}
E.~M.~F. Curado, M.~A. Rego-Monteiro, and Ligia M. C.~S. Rodrigues.
\newblock Structure of generalized {H}eisenberg algebras and quantum
  decoherence analysis.
\newblock {\em Phys. Rev. A}, 87:052120, May 2013.

\bibitem{KMP99}
Ellen Kirkman, Ian~M. Musson, and D.~S. Passman.
\newblock Noetherian down-up algebras.
\newblock {\em Proc. Amer. Math. Soc.}, 127(11):3161--3167, 1999.

\bibitem{KL00}
G\"{u}nter~R. Krause and Thomas~H. Lenagan.
\newblock {\em Growth of algebras and {G}elfand-{K}irillov dimension},
  volume~22 of {\em Graduate Studies in Mathematics}.
\newblock American Mathematical Society, Providence, RI, revised edition, 2000.

\bibitem{lLB95}
Lieven Le~Bruyn.
\newblock Conformal {${\rm sl}_2$} enveloping algebras.
\newblock {\em Comm. Algebra}, 23(4):1325--1362, 1995.

\bibitem{sL17}
Samuel~A. Lopes.
\newblock Non-{N}oetherian generalized {H}eisenberg algebras.
\newblock {\em J.\ Algebra Appl.}, 16(2):1750064, 2017.

\bibitem{LR20arXiv}
Samuel~A. {Lopes} and Farrokh {Razavinia}.
\newblock {Quantum generalized Heisenberg algebras and their representations}.
\newblock {\em arXiv e-prints}, page arXiv:2004.09301, April 2020.

\bibitem{LMZ15}
Rencai L{\"u}, Volodymyr Mazorchuk, and Kaiming Zhao.
\newblock Simple weight modules over weak generalized {W}eyl algebras.
\newblock {\em J. Pure Appl. Algebra}, 219(8):3427--3444, 2015.

\bibitem{LZ15}
Rencai L{\"u} and Kaiming Zhao.
\newblock Finite-dimensional simple modules over generalized {H}eisenberg
  algebras.
\newblock {\em Linear Algebra Appl.}, 475:276--291, 2015.

\bibitem{sR02}
Sonia Rueda.
\newblock Some algebras similar to the enveloping algebra of {$\rm sl(2)$}.
\newblock {\em Comm. Algebra}, 30(3):1127--1152, 2002.

\bibitem{spS90}
S.~P. Smith.
\newblock A class of algebras similar to the enveloping algebra of {${\rm
  sl}(2)$}.
\newblock {\em Trans. Amer. Math. Soc.}, 322(1):285--314, 1990.

\bibitem{jZ97}
James~J. Zhang.
\newblock A note on {GK} dimension of skew polynomial extensions.
\newblock {\em Proc. Amer. Math. Soc.}, 125(2):363--373, 1997.

\end{thebibliography}
\def\cprime{$'$}

\end{document}